\documentclass[draft]{amsart}
\usepackage{amsmath}
\usepackage{amsfonts}
\usepackage{amssymb}
\usepackage{graphicx}
\usepackage{color}
\usepackage[all]{xy}
\providecommand{\U}[1]{\protect\rule{.1in}{.1in}}

\newtheorem{theorem}{Theorem}[section]

\newtheorem{caution}[theorem]{Caution}
\newtheorem{definition}[theorem]{Definition}
\newtheorem{fact}[theorem]{Fact}

\newtheorem{lemma}[theorem]{Lemma}

\newtheorem{remark}[theorem]{Remark}

\begin{document}

\title[Comparison of Abelian 2-categories]{Comparison of the definitions of Abelian 2-categories}

\author{Hiroyuki NAKAOKA }

\address{Graduate School of Mathematical Sciences, The University of Tokyo 
3-8-1 Komaba, Meguro, Tokyo, 153-8914 Japan}

\email{deutsche@ms.u-tokyo.ac.jp}
\thanks{The author wishes to thank Dr. Mathieu Dupont, for pointing out the author's misunderstanding of Definition 165 in \cite{D}.}

\maketitle

\begin{abstract}
In the efforts to define a 2-categorical analog of an abelian category, two (or three) notions of \lq\lq abelian 2-categories" are defined in \cite{N} and \cite{D}.
One is the {\it relatively exact 2-category} defined in \cite{N}, and the other(s) is the {\it (2-)abelian} $\mathrm{Gpd}${\it -category} defined by Dupont \cite{D}.
We compare these notions, using the arguments in \cite{N} and \cite{D}.
Since they proceed independently in their own way, in different settings and terminologies, it will be worth while to collect and unify them.
In this paper, by comparing their definitions and arguments, we show the relationship among these classes of 2-categories.
\end{abstract}

\section{introduction}
Motiveted by \cite{RMV}, we defined a general class of 2-categories \lq {\it relatively exact 2-categories}' in \cite{N} (originally written as our master's thesis in 2006), so as to make the 2-categorical homological algebra work well in an abstract setting.

A relatively exact 2-category is a generalization of $\mathrm{SCG}$ ($=$ the 2-category of symmetric categorical groups), and defined as a 2-categorical analog of an abelian category.
\begin{center}
\begin{tabular}
[c]{|c|c|c|}\hline
& category & 2-category\\\hline
general theory & abelian category & relatively exact 2-category\\\hline
example & $\mathrm{Ab}$ & $\mathrm{SCG}$\\\hline
\end{tabular}
\end{center}
On the other hand, with a similar motivation, Dupont defined two classes of 2-categries \lq{\it 2-abelian $\mathrm{Gpd}$-category}' and \lq{\it abelian $\mathrm{Gpd}$-category}' in \cite{D}.
Thus there are three different classes of 2-categories\\
$\bullet$ (Relatively exact 2-category)\\
$\bullet$ (2-abelian $\mathrm{Gpd}$-category)\\
$\bullet$ (abelian $\mathrm{Gpd}$-category)\\
defined as 2-dimensional analogs of abelian categories.
So it will be necessary to make explicit the relations.

We compare these notions, using the arguments in \cite{N} and \cite{D}.
Since they proceed independently in their own way, in different settings and terminologies, it will be worth while to collect and unify them.

In this paper, by comparing their definitions and arguments, we show the relationship among three classes of 2-categories mentioned above. In Theorem \ref{CompThm}, we show there are implications for these notions
\[ \mathrm{(2\text{-}Abelian\ Gpd)}\Rightarrow\mathrm{(Relatively\ exact)}\Rightarrow\mathrm{(Abelian\ Gpd)}, \]
except for some minor differences (see Theorem \ref{CompThm}).

\section{Preliminaries}

Let $\mathbf{S}$ denote a 2-category (in the strict sense). We use the following notation.\\
$\mathbf{S}^0$, $\mathbf{S}^1$, $\mathbf{S}^2:$ class of 0-cells, 1-cells, and 2-cells in $\mathbf{S}$, respectively.\\
$\mathbf{S}^1(A,B):$ 1-cells from $A$ to $B$, where $A,B\in\mathbf{S}^0$.\\
$\mathbf{S}^2(f,g):$ 2-cells from $f$ to $g$, where $f,g\in\mathbf{S}^1(A,B)$ for certain $A,B\in\mathbf{S}^0$.\\
$\mathbf{S}(A,B):$ $\mathrm{Hom}$-category between $A$ and $B$\\
$($i.e. $\mathrm{Ob}(\mathbf{S}(A,B))=\mathbf{S}^1(A,B)$, $\mathbf{S}(A,B)(f,g)=\mathbf{S}^2(f,g))$.

In diagrams, $\longrightarrow$ represents a 1-cell, $\Longrightarrow$ represents a 2-cell, $\circ$ represents a horizontal composition, and $\cdot$ represents a vertical composition. We use capital letters $A,B,\ldots$ for 0-cells, small letters $f,g,\ldots$ for 1-cells, and Greek symbols $\alpha,\beta,\ldots$ for 2-cells.

The composition of $A\overset{f}{\longrightarrow}B$ and $B\overset{g}{\longrightarrow}C$ is denoted by $g\circ f$, conversely to \cite{N}. Similarly for the composition of 2-cells.

In the following arguments, any 2-cell in a 2-category is invertible. 
This helps us to avoid being fussy about the directions of 2-cells, and we use the word \lq dual' simply to reverse the directions of 1-cells.
For example, {\it cokernel} is the dual notion of {\it kernel}, and {\it pullback} is dual to {\it pushout}. As for the definitions of (co-)kernels, pullbacks, and pushouts in a 2-category, see \cite{D} or \cite{N}. (The definitions in \cite{N} and \cite{D} agree.)

\section{Relatively exact 2-category}

Let $\mathrm{SCG}$ denote the 2-category of small symmetric categorical groups ($=$ symmetric 2-groups). This is denoted by $\mathrm{2\text{-}SGp}$ in \cite{D}.
0-cells are symmetric categorical groups, 1-cells are symmetric monoidal functors, and 2-cells are monoidal transformations (cf. \cite{RMV} or \cite{N}).

For any symmetric monoidal functor $f:A\rightarrow B$, let $f_I$ denote the unit isomorphism $f_I:f(0_A)\overset{\cong}{\longrightarrow} 0_B$, where $0_A$ and $0_B$ are respectively the unit of $A$ and $B$, with respect to $\otimes$.

\begin{definition}{\rm $($Definition 3.7 in \cite{N}$)$
\label{LocSCG}
A 2-category $\mathbf{S}$ is said to be {\it locally} $\mathit{SCG}$ if the following conditions are satisfied$:$\\
{\rm (LS1)} For every $A,B\in\mathbf{S}^0$, there is a given functor $\otimes_{A,B}:\mathbf{S}(A,B)\times\mathbf{S}(A,B)\rightarrow\mathbf{S}(A,B)$, and an object $0_{A,B}\in \mathrm{Ob}(\mathbf{S}(A,B))$ such that $(\mathbf{S}(A,B),\otimes_{A,B},0_{A,B})$ becomes a symmetric categorical group, and the following naturality conditions are satisfied$:$
\[ 0_{A,B}\circ0_{B,C}=0_{A,C}\quad\quad(\forall A,B,C\in\mathbf{S}^0) \]
\noindent{\rm (LS2)} $\mathrm{Hom}=\mathbf{S}(-,-):\mathbf{S}^{\mathrm{op}}\times\mathbf{S}\rightarrow\mathrm{SCG}$ is a 2-functor (in the strict sence).

Moreover, for any $A,B,C\in\mathbf{S}^0$,
\begin{eqnarray}
(-\circ 0_{A,B})_I&=\mathrm{id}_{0_{A,C}}\in\mathbf{S}^2(0_{A,C},0_{A,C})\label{1-6}\\
(0_{B,C}\circ -)_I&=\mathrm{id}_{0_{A,C}}\in\mathbf{S}^2(0_{A,C},0_{A,C}).\label{1-7}
\end{eqnarray}
are satisfied.
\[
\xy
(-16,0)*+{A}="0";
(0,4)*+{B}="2";
(16,0)*+{C}="4";
{\ar^{0_{A,B}} "0";"2"};
{\ar^{0_{B,C}} "2";"4"};
{\ar@/_1.20pc/_{0_{A,C}} "0";"4"};
{\ar@{=>}_{\mathrm{id}} (0,1)*{};(0,-3)*{}};
\endxy
\]
(Remark that $(-\circ 0_{A,B})$ and $(0_{B,C}\circ -)$ are symmetric monoidal functors.)

\noindent{\rm (LS3)} There is a 0-cell $0\in\mathbf{S}^0$ called {\it zero object}, which satisfies the following conditions:\\
\ \ \ {\rm (ls3-1)} For any $f:0\rightarrow A$ in $\mathbf{S}$, there exists a unique 2-cell $\theta_f\in\mathbf{S}^2(f,0_{0,A})$.\\
\ \ \ {\rm (ls3-2)} For any $f:A\rightarrow 0$ in $\mathbf{S}$, there exists a unique 2-cell $\tau_f\in\mathbf{S}^2(f,0_{A,0})$.\\
{\rm (LS3$+$)} $\mathbf{S}(0,0)$ is the zero categorical group.\\
{\rm (LS4)} For any $A,B\in\mathbf{S}^0$, their product and coproduct exist.
}\end{definition}

\begin{caution}\label{ZeroCaution}{\rm
In \cite{N}, {\it zero object} was also assumed to satisfy {\rm (LS3$+$)}.
On the other hand, the definition of zero object in \cite{D} only requires {\rm (ls3-1)} and{\rm (ls3-2)}.
In fact, condition {\rm (LS3$+$)} is not used essentially in \cite{N}. So in the following, we mainly consider lically $\mathrm{SCG}$ 2-categories without condition {\rm (LS3$+$)}.
}
\end{caution}

\begin{definition}\label{RE}{\rm (Definition 3.7 in \cite{N})
Let $\mathbf{S}$ be a locally $\mathrm{SCG}$ 2-category. $\mathbf{S}$ is said to be {\it relatively exact} if the following conditions are satisfied$:$\\
{\rm (RE1)} For any 1-cell $f$, its kernel and cokernel exist.\\
{\rm (RE2)} Any 1-cell $f$ is faithful if and only if $f=\mathrm{ker}(\mathrm{cok}(f))$.\\
{\rm (RE3)} Any 1-cell $f$ is cofaithful if and only if $f=\mathrm{cok}(\mathrm{ker}(f))$.
}
\end{definition}
(For the definitions of (fully) (co-) faithfulness, see \cite{N} or \cite{DV}.)

\begin{remark}\label{RemKer}{\rm
For any 1-cell $f:A\rightarrow B$, its kernel is defined as the triplet $(\mathrm{Ker}(f),\mathrm{ker}(f),\varepsilon_f)$
\[
\xy
(-20,0)*+{\mathrm{Ker}(f)}="0";
(0,0)*+{A}="2";
(20,0)*+{B}="4";
{\ar_{\mathrm{ker}(f)} "0";"2"};
{\ar_{f} "2";"4"};
{\ar@/^1.60pc/^{0} "0";"4"};
{\ar@{=>}_{\varepsilon_f} (0,2)*{};(0,6)*{}};
\endxy
,
\]
universal among those $(K,k,\varepsilon)$
\[
\xy
(-16,0)*+{K}="0";
(0,0)*+{A}="2";
(16,0)*+{B}="4";
{\ar_{k} "0";"2"};
{\ar_{f} "2";"4"};
{\ar@/^1.60pc/^{0} "0";"4"};
{\ar@{=>}_{\varepsilon} (0,2)*{};(0,6)*{}};
\endxy
.
\]
For the precise definition, see \cite{N} or \cite{D}.
Dually, the cokernel of $f$ is the universal triplet $(\mathrm{Cok}(f),\mathrm{cok}(f),\pi_f)$
\[
\xy
(-20,0)*+{A}="0";
(0,0)*+{B}="2";
(20,0)*+{\mathrm{Cok}(f)}="4";
{\ar_{f} "0";"2"};
{\ar_{\mathrm{cok}(f)} "2";"4"};
{\ar@/^1.60pc/^{0} "0";"4"};
{\ar@{=>}_{\pi_f} (0,2)*{};(0,6)*{}};
\endxy
.
\]
}
\end{remark}

\section{(2-)Abelian $\mathrm{Gpd}$-category}

(2-)Abelian $\mathrm{Gpd}$-categories, defined in \cite{D}, are $\mathrm{Gpd}^{\ast}$-categories satisfying certain conditions.
By definition, a $\mathrm{Gpd}^{\ast}$-category is a category $\mathcal{C}$ enriched by the category $\mathrm{Gpd}^{\ast}$ of small pointed groupoids (Proposition 70 in \cite{D}).
For any $A,B\in\mathrm{Ob}(\mathcal{C})$, the distinguished point in $\mathcal{C}(A,B)$ is denoted by $0_{A,B}$ or simply by $0$.

In \cite{D}, it is remarked that any $\mathrm{Gpd}^{\ast}$-category $\mathcal{C}$ is equivalent to a {\it strictly described} one, and thus $\mathcal{C}$ is assumed to be strictly described, namely, it satisfies the following:\\
{\rm (SD1)} For any sequence of morphisms
\[ A\overset{f}{\longrightarrow}B\overset{g}{\longrightarrow}C\overset{h}{\longrightarrow}D \]
in $\mathcal{C}$,
\[ h\circ(g\circ f)=(h\circ g)\circ f \]
is satisfied.\\
{\rm (SD2)} For any $f:A\rightarrow B$ in $\mathcal{C}$,
\begin{eqnarray*}
f\circ\mathrm{id}_A=f\\
\mathrm{id}_B\circ f=f
\end{eqnarray*}
are satisfied.\\
{\rm (SD3)} For any $f:A\rightarrow B$ and any objects $A^{\prime}, B^{\prime}$ in $\mathcal{C}$,
\begin{eqnarray*}
f\circ 0_{A^{\prime},A}=0_{A^{\prime},B}\\
0_{B,B^{\prime}}\circ f=0_{A,B^{\prime}}
\end{eqnarray*}
are satisfied.\\
{\rm (SD4)} For any %
$\xy%
(-8,0)*+{A}="0";%
(8,0)*+{B}="2";%
{\ar@/^1.00pc/^{f} "0";"2"};%
{\ar@/_1.00pc/_{g} "0";"2"};%
{\ar@{=>}_{\alpha} (0,2)*{};(0,-2)*{}};%
\endxy%
$ %
and any objects $A^{\prime}, B^{\prime}$ in $\mathcal{C}$,
\begin{eqnarray*}
\alpha\circ0_{A^{\prime},A}=\mathrm{id}_{0_{A^{\prime},B}}\\
0_{B,B^{\prime}}\circ\alpha=\mathrm{id}_{0_{A,B^{\prime}}}
\end{eqnarray*}
are satisfied.

\begin{remark}{\rm
A $\mathrm{Gpd}^{\ast}$-category $\mathcal{C}$ is regarded as a 2-category in the following, and we use 2-categorical terminologies, e.g. \lq 0-cell' for an object, \lq 1-cell' for an arrow.
}\end{remark}

\begin{definition}\label{AG}{\rm (Definition 165 in \cite{D})
An {\it abelian $\mathit{Gpd}$-category} is a $\mathrm{Gpd}^{\ast}$-category $\mathcal{C}$ with zero object, finite $($co-$)$products and $($co-$)$kernels, satisfying the following conditions$:$\\
{\rm (AG1)} Every 0-monomorphic 1-cell $f$ satisfies $f=\mathrm{ker}(\mathrm{cok}(f))$.\\
{\rm (AG2)} Every 0-epimorphic 1-cell $f$ satisfies $f=\mathrm{cok}(\mathrm{ker}(f))$.\\
{\rm (AG3)} Fully 0-faithful 1-cells and 0-monomorphic 1-cells are stable under pushout.\\
{\rm (AG4)} Fully 0-cofaithful 1-cells and 0-epimorphic 1-cells are stable under pullback.
}\end{definition}

\begin{definition}\label{2AG}{\rm (Definition 179 and 183 in \cite{D})
A {\it 2-abelian $\mathit{Gpd}$-category} is a $\mathrm{Gpd}^{\ast}$-category $\mathcal{C}$ with zero object, finite $($co-$)$products and $($co-$)$kernels, satisfying the following conditions$:$\\
{\rm (2AG1)} If $f$ is a 0-faithful 1-cell, then $f=\mathrm{ker}(\mathrm{cok}(f))$.\\
{\rm (2AG2)} If $f$ is a 0-cofaithful 1-cell, then $f=\mathrm{cok}(\mathrm{ker}(f))$.\\
{\rm (2AG3)} Any fully 0-faithful 1-cell is canonically the root of its copip.\\
{\rm (2AG4)} Any fully 0-cofaithful 1-cell is canonically the coroot of its pip.
}\end{definition}

For the definitions of (co-)roots and (co-)pips, see \cite{D}. We do not require them explicitly in the following arguments. We introduce the rest of the notions appearing in the above definitions.
The definition of 0-monomorphic 1-cells is the following. 0-epimorphicity is defined dually.
\begin{definition}{\rm (Definition 118 in \cite{D})
A 1-cell $f:A\rightarrow B$ is {\it 0-monomorphic} if, for any 1-cell $a:X\rightarrow A$ and any 2-cell $\beta:f\circ a\Longrightarrow 0$ compatible with $\pi_f$ (of Remark \ref{RemKer}), there exists a unique $\alpha:a\Longrightarrow 0$ such that $f\circ\alpha=\beta$.
}
\end{definition}

The definitions of (fully) 0-faithful 1-cells are the following. (Fully) 0-cofaithful 1-cells are defined dually.
\begin{definition}{\rm (Definition 78, 80 in \cite{D})
Let $\mathcal{C}$ be a $\mathrm{Gpd}^{\ast}$-category, and $f:A\rightarrow B$ be a 1-cell in $\mathcal{C}$.

\noindent{\rm (i)} $f$ is {\it 0-faithful} if for any $\xy%
(-8,0)*+{X}="0";%
(8,0)*+{A}="2";%
{\ar@/^1.00pc/^{0} "0";"2"};%
{\ar@/_1.00pc/_{0} "0";"2"};%
{\ar@{=>}_{\alpha} (0,2)*{};(0,-2)*{}};%
\endxy%
$ %
in $\mathcal{C}$,
\[ f\circ\alpha=\mathrm{id}_0\ \ \ \Rightarrow\ \ \ \alpha=\mathrm{id}_0 \]
is satisfied.

\noindent{\rm (ii)} $f$ is {\it fully 0-cofaithful} if for any 1-cell $a:X\rightarrow A$ and any 2-cell $\xy%
(-8,0)*+{X}="0";%
(8,0)*+{B}="2";%
{\ar@/^1.00pc/^{f\circ\alpha} "0";"2"};%
{\ar@/_1.00pc/_{0} "0";"2"};%
{\ar@{=>}_{\alpha} (0,2)*{};(0,-2)*{}};%
\endxy%
$%
,
there exists a unique 2-cell $\alpha:a\Longrightarrow 0$ such that $\beta=f\circ\alpha$.
}
\end{definition}

\begin{fact}{\rm 
In \cite{D}, it is shown that any 2-abelian $\mathrm{Gpd}$-category $\mathcal{C}$ admits a weak enrichment by $\mathrm{SCG}$, i.e., $\mathcal{C}$ is {\it preadditive}, in the terminology of \cite{D}.
}
\end{fact}

For the general definition of a preadditive $\mathrm{Gpd}$-category, see \cite{D}. We only consider the case where $\mathcal{C}$ is strictly described.
(In this case, the natural transformations appearing in Definition 218 in \cite{D} are identities)

\begin{definition}{\rm 
A strictly described $\mathrm{Gpd}$-category $\mathcal{C}$ is {\it preadditive} if it satisfies the following:

\noindent{\rm (o)} For any 0-cells $A,B$ in $\mathcal{C}$, $\mathrm{Hom}$-category $\mathcal{C}(A,B)$ is equipped with a structure of a symmetric categorical group $(\mathcal{C}(A,B),\otimes,0)$.

\noindent{\rm (a1)} For any 1-cell $A\overset{f}{\longrightarrow}B$ and any 0-cell $C$ in $\mathcal{C}$, the composition by $f$
\[ -\circ f:\mathcal{C}(B,C)\rightarrow\mathcal{C}(A,C) \]
is symmetric monoidal.

\noindent{\rm (a2)} The dual of {\rm (a1)}.

\noindent{\rm (b1)} For any 1-cells $A\overset{f}{\longrightarrow}B\overset{g}{\longrightarrow}C$ and any 0-cell $D$ in $\mathcal{C}$, we have
\[
\xy
(-12,7)*+{\mathcal{C}(C,D)}="0";
(12,7)*+{\mathcal{C}(B,D)}="2";
(0,-8)*+{\mathcal{C}(A,D)}="4";
(0,8)*+{}="6";
{\ar^{-\circ g} "0";"2"};
{\ar_{-\circ (g\circ f)} "0";"4"};
{\ar^{-\circ f} "2";"4"};
{\ar@{}|\circlearrowright"6";"4"};
\endxy
\]
as monoidal functors.

\noindent{\rm (b2)} The dual of {\rm (b1)}.

\noindent{\rm (c)} For any 1-cells $A\overset{f}{\longrightarrow}B$ and $C\overset{g}{\longrightarrow}D$, we have
\[
\xy
(-12,6)*+{\mathcal{C}(B,C)}="0";
(12,6)*+{\mathcal{C}(A,C)}="2";
(-12,-6)*+{\mathcal{C}(B,D)}="4";
(12,-6)*+{\mathcal{C}(A,D)}="6";
{\ar^{-\circ f} "0";"2"};
{\ar_{g\circ -} "0";"4"};
{\ar^{g\circ -} "2";"6"};
{\ar_{-\circ f} "4";"6"};
{\ar@{}|\circlearrowright"0";"6"};
\endxy
\]
as monoidal functors.

\noindent{\rm (d)} For any 0-cells $A$ and $B$ in $\mathcal{C}$, we have
\[
\xy
(-12,0)*+{\mathcal{C}(A,B)}="0";
(12,0)*+{\mathcal{C}(A,B)}="2";
(0,3)*+{_{\circlearrowright}}="6";
(0,-3)*+{_{\circlearrowright}}="8";
{\ar@/^1.20pc/^{\mathrm{id}_B\circ -} "0";"2"};
{\ar@/_1.20pc/_{-\circ\mathrm{id}_A} "0";"2"};
{\ar|{\mathrm{id}} "0";"2"};
\endxy
\]

\noindent{\rm (e1)} For any 0-cell $X$ and any $\ell,k:X\rightarrow A$,
\[ (-\circ f)_{\ell,k}:(\ell\otimes k)\circ f\Longrightarrow (\ell\circ f)\otimes(k\circ f)\ \ \ (\forall f:A\rightarrow B) \]
is natural in $f$.
\[
\xy
(-16,0)*+{X}="0";
(-14,1)*+{}="01";
(-14,-1)*+{}="02";
(0,0)*+{A}="2";
(-2,1)*+{}="21";
(-2,-1)*+{}="22";
(14,0)*+{B}="4";
{\ar^{\ell} "01";"21"};
{\ar_{k} "02";"22"};
{\ar^{f} "2";"4"};
\endxy
\]

\noindent{\rm (e2)} The dual of {\rm (e1)}.

\noindent{\rm (f1)} For any $X$,
\[ (-\circ f)_I:0_{X,A}\circ f\Longrightarrow 0_{X,B}\ \ \ (\forall f:A\rightarrow B) \]
is natural in $f$.

\noindent{\rm (f2)} The dual of {\rm (f1)}.

Here, since $(-\circ f)$ is monoidal, $(-\circ f)_{\ell,k}$ denotes the structure isomorphism
\[ (-\circ f)_{\ell,k}:(\ell\otimes k)\circ f\Longrightarrow (\ell\circ f)\otimes(k\circ f) \]
natural in $\ell,k:X\rightarrow A$.

Similarly, $(-\circ f)_I$ denotes the unit isomorphism.
}
\end{definition}

\section{Comparison}

\begin{lemma}\label{LemPre}{\rm 
If $\mathcal{C}$ is a strictly described preadditive $\mathrm{Gpd}$-category, then 
\[ \mathrm{Hom}=\mathcal{C}(-,-):\mathcal{C}^{\mathrm{op}}\times\mathcal{C}\rightarrow\mathrm{SCG} \]
is a 2-functor.
}
\end{lemma}
\begin{proof}
For the definition of a 2-functor, see Definition 7.2.1 in \cite{B}.
It can be easily shown that, to show the lemma, it suffices to show the following conditions:

\noindent{\rm (i)} For any 1-cells $A^{\prime}\overset{f}{\longrightarrow}A$ and $B\overset{g}{\longrightarrow}B^{\prime}$,
\[ g\circ -\circ f:\mathcal{C}(A,B)\rightarrow \mathcal{C}(A^{\prime},B^{\prime}) \]
is a symmetric monoidal functor.

\noindent{\rm (ii)} For any 2-cells $\xy%
(-8,0)*+{A^{\prime}}="0";%
(8,0)*+{A}="2";%
{\ar@/^1.00pc/^{f} "0";"2"};%
{\ar@/_1.00pc/_{f^{\prime}} "0";"2"};%
{\ar@{=>}_{\alpha} (0,2)*{};(0,-2)*{}};%
\endxy%
$
and
$\xy%
(-8,0)*+{B^{\prime}}="0";%
(8,0)*+{B}="2";%
{\ar@/^1.00pc/^{g} "0";"2"};%
{\ar@/_1.00pc/_{g^{\prime}} "0";"2"};%
{\ar@{=>}_{\beta} (0,2)*{};(0,-2)*{}};%
\endxy%
$,
\[ \beta\circ-\circ\alpha:g\circ-\circ f\Longrightarrow g^{\prime}\circ-\circ f^{\prime} \]
is a monoidal transformation.

\noindent{\rm (iii)} For any 1-cells $A^{\prime\prime}\overset{f^{\prime}}{\longrightarrow}A^{\prime}\overset{f}{\longrightarrow}A$ and $B\overset{g}{\longrightarrow}B^{\prime}\overset{g^{\prime}}{\longrightarrow}B^{\prime\prime}$, we have
\[
\xy
(-12,7)*+{\mathcal{C}(A,B)}="0";
(12,7)*+{\mathcal{C}(A^{\prime},B^{\prime})}="2";
(0,-8)*+{\mathcal{C}(A^{\prime\prime},B^{\prime\prime})}="4";
(0,8)*+{}="6";
{\ar^{g\circ-\circ f} "0";"2"};
{\ar_{(g^{\prime}\circ g)\circ-\circ (f^{\prime}\circ f)} "0";"4"};
{\ar^{g^{\prime}\circ-\circ f^{\prime}} "2";"4"};
{\ar@{}|\circlearrowright"6";"4"};
\endxy
\]
as monoidal functors.

\noindent{\rm (iv)} For any 0-cells $A$ and $B$ in $\mathcal{C}$, we have
\[
\xy
(-12,0)*+{\mathcal{C}(A,B)}="0";
(12,0)*+{\mathcal{C}(A,B)}="2";
(0,0)*+{_{\circlearrowright}}="4";
{\ar@/^1.00pc/^{\mathrm{id}\circ-\circ\mathrm{id}} "0";"2"};
{\ar@/_1.00pc/_{\mathrm{id}} "0";"2"};
\endxy
\]
as monoidal functors.

{\rm (iv)} follows from {\rm (d)}.
{\rm (i)} follows from {\rm (a1)}, {\rm (a2)} (and {\rm (c)}).
{\rm (iii)} follows from {\rm (b1)}, {\rm (b2)} (and {\rm (c)}).
{\rm (ii)} follows from {\rm (e1)}, {\rm (e2)}, {\rm (f1)}, {\rm (f2)}.
\end{proof}

\begin{lemma}\label{LemFinal}{\rm
Let $f:A\rightarrow B$ be any 1-cell in a relatively exact 2-category. Then the following are satisfied.

\noindent{\rm (i)} $f$ is faithful if and only if it is 0-faithful, if and only if it is 0-monomorphic.

\noindent{\rm (ii)} $f$ is fully faithful if and only if it is fully 0-faithful.}
\end{lemma}
\begin{proof}
{\rm (i)} By Corollary 3.24 in \cite{N}, $f$ is faithful if and only if it is 0-faithful.
As remarked after Definition 118 in \cite{D}, any 0-monomorphic 1-cell is faithful.
Conversely, if $f$ is faithful, then $f$ satisfies $f=\mathrm{ker}(\mathrm{cok}(f))$, and becomes 0-monomorphic by Lemma 3.19 in \cite{N}.

{\rm (ii)} This is nothing other than Lemma 3.22 (2) in \cite{N}.

\end{proof}

\begin{theorem}\label{CompThm}
{\rm There are the implications among the conditions on 2-categories
\[ (\text{2-Abelian}\ \mathrm{Gpd})\ \Rightarrow\ (\text{Relatively exact})\ \Rightarrow \ (\text{Abelian}\ \mathrm{Gpd}). \]
More precisely, we have:

\noindent{\rm (i)} Any strictly described 2-abelian $\mathrm{Gpd}$-category is a relatively exact 2-category without condition {\rm (LS3$+$)}.

\noindent{\rm (ii)} Any relatively exact 2-category without condition {\rm (LS3$+$)} is an abelian $\mathrm{Gpd}$-category not necessarily strictly described.

}
\end{theorem}
\begin{proof}
First remark that each of these 2-categories is a $\mathrm{Gpd}^{\ast}$-category with zero object, finite (co-)products and (co-)kernels.

{\rm (i)}
Let $\mathcal{C}$ be a 2-abelian $\mathrm{Gpd}$-category.
$\mathcal{C}$ satisfies {\rm (LS1)}, as a particular case of {\rm (SD3)}.
By Lemma \ref{LemPre}, $\mathrm{Hom}=\mathcal{C}(-,-):\mathcal{C}^{\mathrm{op}}\times\mathcal{C}\rightarrow\mathrm{SCG}$ is a 2-functor.
Moreover, $(\ref{1-6})$ and $(\ref{1-7})$ in {\rm (LS2)} follows from {\rm (SD4)}.
Thus  $\mathcal{C}$ satisfies {\rm (LS2)}. 
By Proposition 180 in \cite{D}, any 1-cell in $\mathcal{C}$ is 0-faithful if and only if it is faithful. Thus {\rm (RE2)} follows from {\rm (2AG1)}. Dually, {\rm (RE3)} follows from {\rm (2AG2)}.

{\rm (ii)}
Let $\mathbf{S}$ be a relatively exact 2-category. By the duality, it suffices to show 
{\rm (AG1)} and {\rm (AG3)}.
By Lemma \ref{LemFinal}, we have equivalences of the notions
\begin{center}
faithful $=$ 0-faithful $=$ 0-monomorphic\\
fully faithful $=$ fully 0-faithful
\end{center}
for 1-cells in $\mathbf{S}$. Thus {\rm (AG1)} follows from {\rm (RE2)}, and {\rm (AG3)} follows from the duals of Proposition 3.32 and Proposition 5.12 in \cite{N}.

Remark also that {\rm (SD1)} and {\rm (SD2)} are satisfied, but {\rm (SD3)} and {\rm (SD4)} are not satisfied in general. So $\mathbf{S}$ is not necessarily strictly described as a $\mathrm{Gpd}^{\ast}$-category.
\end{proof}



\end{document}